\documentclass[a4paper, 9pt]{article}
\usepackage{amsmath}
\usepackage{amssymb}
\usepackage{amsthm}
\usepackage{titlesec}
\titleformat*{\section}{\normalsize\bfseries}
\usepackage[margin=25mm]{geometry}
\def\R{\mathbb{R}}

\def\Z{\mathbb{Z}}
\def\C{\mathbb{C}}
\def\Re{{\mathrm{Re}}}
        
\def\a{\alpha}
\def\b{\beta}

\def\d{\displaystyle}
\def\coss{{\mathrm{Cos^{-1}}}}

\newtheorem{thm}{Theorem}[section]
\newtheorem{lem}[thm]{Lemma}
\newtheorem{cor}[thm]{Corollary}
\newtheorem{prop}[thm]{Proposition}
\newtheorem{rem}[thm]{Remark}

\makeatletter
 \@addtoreset{equation}{section}
 
 \makeatother
 
\begin{document}

\title{\large{\bf{ASYMPTOTIC STABILITY AND STABILITY SWITCHING \\FOR A SYSTEM OF DELAY DIFFERENTIAL EQUATIONS}}}

\author{Wataru Saito and Ikki Fukuda\\ [.7em]
Department of Mathematics, Hokkaido University}

\date{}

\maketitle

\footnote[0]{2010 Mathematics Subject Classification: 34K20, 34K25, 93C23.}

\begin{abstract}
In this paper, we consider the asymptotic stability for a system of linear delay differential equations. By analysing of the characteristic equation in detail, we have established the necessary and sufficient condition for the asymptotic stability for the zero solution of the system including the ``stability switching'' which describe the transition between stability and instability.
\end{abstract}

\smallskip
{\bf Keywords:} asymptotic stability, characteristic equation, delay, stability switching. 

\section{Introduction}
In this paper, we consider the following linear system of delay differential equations: 
\begin{align}
\begin{split}
x'(t)+ax(t)+\a x(t-\tau)+by(t)&=0, \\
y'(t)+ay(t)+cx(t)+\a y(t-\tau)&=0, \\
\end{split}
\end{align}
where $x(t)$ and $y(t)$ denote real-valued unknown functions for $t\ge 0$. The coefficients $a$, $\a$, $b$ and $c$ are real numbers, while $\tau$ is a positive number, called the time delay parameter. To study delay differential equations is very important because it appears in several physical fields and engineering as mathematical models. For example, population models of Lotka-Volterra type and prey-predator models, neural network models, and also traffic flow, see e.g.  \cite{1, 5, 17, 20}. In the present paper, we consider the asymptotic stability for the zero solution of (1.1). 

Let us introduce the known results related to this problem briefly. For the scalar equation, the asymptotic stability is well studied (e.g. \cite{22, 4, 21}). On the other hand, Suzuki-Matsunaga \cite{19} studied some system of delay differential equations with off-diagonal delay. Also, Matsunaga \cite{13} studied the system (1.1) with $a=0$, and obtained the necessary and sufficient condition for the asymptotic stability. Moreover, Nishiguchi \cite{24} generalized the result obtained in \cite{13}. In this paper, we derive the necessary and sufficient condition for the asymptotic stability for the zero solution of (1.1). Especially, the main purpose of this paper is to derive the condition associated with stability switching. 
\section{Main Result}
Let us state our main result concerning the necessary and sufficient condition for the asymptotic stability of the system (1.1): 
\begin{thm}
Suppose that $\a^{2}\neq a^{2}$ when $bc<0$. The zero solution of the system $(1.1)$ is asymptotically stable if and only if any one of the following conditions holds: 
\\[.3em]
\noindent
$\rm{(i)}$ $bc\ge 0$ and
\begin{equation*}
-(a+\b)<\a \le a+\b \ \ or \ \ 
\begin{cases}
|a+\b|<\a, \\
\d 0<\tau <\frac{1}{\sqrt{\a^{2}-(a+\b)^{2}}} \coss\left(-{\frac{a+\b}{\a}}\right)
\end{cases}  
\end{equation*}
and 
\begin{equation*}
-(a-\b)<\a \le a-\b \ \ or \ \   
\begin{cases}
|a-\b|<\a, \\
\d 0<\tau <\frac{1}{\sqrt{\a^{2}-(a-\b)^2}} \coss\left(-\frac{a-\b}{\a}\right), 
\end{cases}  
\end{equation*}
where $\beta:=\sqrt{bc}$.
\newpage

\noindent
$\rm{(ii)}$ $bc<0$ and \\[.3em]
$\rm{(ii-1)}$ $\a^{2}>a^{2}$ and 
\begin{equation*}
\begin{cases}
\sqrt{D}\ge d \ ,\ a+\a>0\ \ and \ \ \tau \in (0, r_{1.0}),\\
\sqrt{D}< d 
\begin{cases}
a+\a>0
\begin{cases}
k=0\ \ and \ \ \tau \in (0, r_{1.0}),\\
k\ge1\ \ and \ \ \tau \in (0, r_{1.0}) \cup (r_{2.0}, r_{1.1}) \cup \cdots \cup (r_{2.k-1}, r_{1.k}),\\
\end{cases}\\
a+\a<0\ ,\ l\ge0\ \ and \ \ \tau \in (r_{2.0}, r_{1.0}) \cup \cdots \cup (r_{2.l}, r_{1.l}),\\
\end{cases}\\
\end{cases}
\end{equation*}
$\rm{(ii-2)}$ $\a^{2}<a^{2}$ and $a+\a>0$,\\
where $d:=\sqrt{-bc}$, $D:=\a^{2}-a^{2}$, $\d E:=-\frac{a}{\a}$, $\d k:=\left\lfloor\frac{d\ \coss E}{2\sqrt{D}\pi}\right\rfloor$ and $\d l:=\left\lfloor\frac{\left(d-\sqrt{D}\right)\pi-d\ \coss E}{2\sqrt{D}\pi}\right\rfloor$, 
while 
\begin{equation}
\omega_{1}:=-d-\sqrt{D}\ \ and \ \ r_{1.n}:=
\begin{cases}
\d\frac{1}{\omega_1}\left(-2n\pi-\coss E\right),\ \a>0,\\
\d\frac{1}{\omega_1}\left(-2(n+1)\pi+\coss E\right),\ \a<0
\end{cases}
\end{equation}
for $n=0, 1, 2, \cdots$, and 
\begin{equation}
\omega_{2}:=-d+\sqrt{D}\ \ and \ \ r_{2.n}:=
\begin{cases}
\d\frac{1}{\omega_2}\left(-2(n+1)\pi+\coss E\right),\ \a>0,\\
\d\frac{1}{\omega_2}\left(-2n\pi-\coss E\right),\ \a<0
\end{cases}
\end{equation}
for $d\neq \sqrt{D}$ and $n=0, 1, 2, \cdots$. 
Here, we have used the floor function $\left\lfloor x\right\rfloor:=\max \{n\in \Z \ |\ n\le x\}$.
\end{thm}

\begin{rem}
\rm{
If we take $a=0$ in Theorem 2.1, then we can easily get the same conditions for Theorem 1.1 obtained in \cite{13}.
}
\end{rem}

\begin{rem}
\rm{
The conditions for $\tau$ in $\rm{(ii-1)}$ tells us that as $\tau$ increases monotonously from $0$, the stability of the zero solutions of system (1.1) switches in finite time. This kind of phenomena is called {\it ``stability switching''} (see, e.g. \cite{4, 13}). 
}
\end{rem}

\section{Proof of the Main Theorem}
In this section, we prove Theorem 2.1. The stability for the zero solution of system (1.1) is completely determined by the roots of the associated characteristic equation. Now, let $(\lambda, \phi)\in \C\times \R^{2}$. Substituting $(x(t), y(t))^{T}=e^{\lambda t}\phi$ into (1.1), we have the following eigenvalue problem corresponding to (1.1): 
 \begin{equation*}
 \left((\lambda +\a e^{-\lambda \tau})I+A\right)\phi=0 \ \ \text{with} \ \ A=\left(\begin{array}{ccccc}
      a& b \\
      c &a 
    \end{array}\right).
 \end{equation*}
    Thus, our characteristic equation is given by
\begin{align}
\begin{split}
G(\lambda)
:=& \det \left((\lambda +\a e^{-\lambda \tau})I+A\right)=\left(\lambda+\a e^{-\lambda \tau}+a\right)^2-bc \\
=&\begin{cases}
\left(\lambda +\a e^{-\lambda \tau}+a+\sqrt{bc}\right)\left(\lambda +\a e^{-\lambda \tau}+a-\sqrt{bc}\right),\ \ bc \ge 0,\\
\left(\lambda +\a e^{-\lambda \tau}+a+i\sqrt{-bc}\right)\left(\lambda +\a e^{-\lambda \tau}+a-i\sqrt{-bc}\right),\ \ bc<0
\end{cases}
=0.
\end{split}
\end{align}
The above equation $G(\lambda)=0$ is called a characteristic equation of (1.1), and the solution of the equation is called a characteristic root. It is called the zero solution of system (1.1) is asymptotically stable if and only if all the roots of the characteristic equation lie in the open left half of the complex plane (see, e.g. \cite{7}). Therefore, the goal of our study is to derive the conditions that the real parts of all the characteristic roots are negative. 

First, let $\a=0$. In this case, the characteristic roots are either $\lambda=-a\pm \sqrt{bc}$ if $bc\ge0$ or $\lambda=-a\pm i\sqrt{-bc}$ if $bc<0$. Then, all the roots of (3.1) have negative real parts if and only if one of the following conditions holds:
\begin{equation}
bc\ge0, \ -a+\sqrt{bc}<0\ \ \text{or} \ \ bc<0, \ a>0.
\end{equation}

In the following, we consider (3.1) with $\a \neq0$. First we treat the case of $bc\ge0$. Recalling $\beta=\sqrt{bc}$, (3.1) is reduced to 
\begin{equation*}
\lambda +\a e^{-\lambda \tau}+a+\beta=0\ \ \text{or} \ \ \lambda +\a e^{-\lambda \tau}+a-\beta=0. 
\end{equation*}
For these equations, we have the following lemma (for the proof, see e.g. \cite{4, 21}):
\begin{lem}
Let $p$ and $q$ be real numbers. Then all the roots of the equation $\lambda+p+qe^{-\lambda \tau}=0$ have negative real parts if and only if the following conditions holds:
\begin{equation*}
-p<q\le p
\end{equation*}
or 
\begin{equation*}
|p|<q \ \ and \ \ 0<\tau<\frac{1}{\sqrt{q^{2}-p^{2}}}\coss\left(-\frac{p}{q}\right).
\end{equation*}
\end{lem}
\noindent Therefore, by virtue of this Lemma 3.1, one can immediately obtain the following result: 
\begin{prop}
Let $bc\ge0$. Then all the roots of $(3.1)$ have negative real parts if and only if the condition ${\rm(i)}$ holds. 
\end{prop}

Next, we consider the case of $bc<0$. Recalling $d=\sqrt{-bc}$, we see that all the roots of (3.1) have negative real parts if and only if all the roots of  
\begin{equation}
F(\lambda):=\lambda+\a e^{-\lambda \tau}+a+id=0
\end{equation}
have negative real parts since the function $G(\lambda)$ is written as $G(\lambda)=F(\lambda)\overline{F(\overline{\lambda})}$, where $\overline{\lambda}$ denotes the complex conjugate of any complex number $\lambda$. The next lemma plays an essential role in determining the location of roots of (3.3) for a given $\tau$ (for the proof, see \cite{23}).
\begin{lem}
Let $f(\lambda, \tau)=\lambda^{n}+g(\lambda, \tau)$, where $g(\lambda, \tau)$ is an analytic function and $n$ is a positive integer. Assume that $\limsup_{\Re \lambda>0, \ |\lambda|\to \infty}|\lambda^{-n}g(\lambda, \tau)|<1$. As $\tau$ varies, the number of roots of $f(\lambda, \tau)=0$ including multiplicity in the open right half-plane can change only if a root appears on or crosses the imaginary axis. 
\end{lem}
\begin{cor}
As $\tau$ varies, the number of roots of $(3.3)$ including multiplicity in the open right half-plane can change only if a root appears on or crosses the imaginary axis. 
\end{cor}
Consequently, it is important to find the value $\tau$ for which equation (3.3) has root on the imaginary axis. Let $F(i\omega)=0$. Then, we have $F(i\omega)=i\omega+\a e^{-i\omega \tau}+a+id=0$. We divide it into the real part and imaginary part, and obtain $\a \cos \omega \tau=-a$ and $\a \sin \omega \tau =\omega+d$. By squaring both sides of these equality and adding them together, we have $\a^{2}=a^{2}+(\omega+d)^{2}$. Therefore, we obtain $\omega=-d\pm \sqrt{D}$ if $D=\a^{2}-a^{2}\ge0$. On the other hand, if $\a^{2}<a^{2}$, (3.3) has no purely imaginary roots. 
\medskip

In the latter case, we have the following result: 
\begin{prop}
Let $\a \neq0$, $bc<0$ and $\a^{2}<a^{2}$. Then all the roots of $(3.1)$ have negative real parts if and only if $a+\a>0$.   
\end{prop}
\begin{proof}
Under these assumptions, the equation (3.3) has no purely imaginary roots, as we have observed in the above. Since (3.3) has only root $\lambda=-\a-a-id$ for $\tau=0$, Corollary 3.4 implies the conclusion. 
\end{proof}

In what follows, we consider the case where $\a \neq0$, $bc<0$ and $\a^{2}>a^{2}$. First, we show the following lemma: 
\begin{lem}
Let $\a \neq0$, $bc<0$ and $\a^{2}\ge a^{2}$. If $\lambda=i\omega$ $(\omega \in \R)$ and $\tau \ge0$ satisfy $(3.3)$, we have $\omega=\omega_{1}$ and $\tau=r_{1. n}$, or $\omega=\omega_{2}$ and $\tau=r_{2. n}$, where $\omega_{1}$, $\omega_{2}$, $r_{1. n}$ and $r_{2. n}$ are defined by $(2.1)$ and $(2.2)$. Conversely, if we take $\omega=\omega_{1}$ and $\tau=r_{1. n}$, or $\omega=\omega_{2}$ and $\tau=r_{2. n}$, then $\lambda=i\omega$ and $\tau$ satisfy $(3.3)$.    
\end{lem}
\begin{proof}
Let $\lambda=i\omega$ $(\omega \in \R)$ and $\tau \ge0$ satisfy $(3.3)$. From the above observation, we have $\omega=\omega_{1}$ or $\omega=\omega_{2}$. We note that 
\begin{equation}
\cos \omega \tau=-\frac{a}{\a}=E \ \ \text{and} \ \ \sin \omega \tau =\frac{\omega+d}{\a}.
\end{equation}
Then, if $\omega=\omega_{1}=-d-\sqrt{D}$ and $\a>0$, since $\sin \omega_{1}\tau<0$ from (3.4), we have $\omega_{1}\tau=2n\pi-\coss E$. Therefore, we get
\begin{equation*}
\tau=\frac{1}{\omega_{1}}\left(-2n\pi-\coss E\right)=r_{1. n}, \ \ \a>0, \ n=0, 1, 2, \cdots.
\end{equation*}
On the other hand, if $\omega=\omega_{1}=-d-\sqrt{D}$ and $\a<0$, we have 
\begin{equation*}
\tau=\frac{1}{\omega_{1}}\left(-2(n+1)\pi+\coss E\right)=r_{1. n}, \ \ \a<0, \ n=0, 1, 2, \cdots,
\end{equation*}
since $\sin \omega_{1}\tau>0$ from (3.4). Thus we obtain $\tau=r_{1. n}$. If $\omega=\omega_{2}=-d+\sqrt{D}$, similarly we have $\tau=r_{2. n}$. The second statement can be proved by a direct calculation.
\end{proof}

Next we examine how the roots of (3.3) cross the imaginary axis as $\tau$ increases. 
\begin{lem}
Let $\a \neq0$, $bc<0$ and $\a^{2}>a^{2}$. Then the purely imaginary root $i\omega_{1}$ of $(3.3)$ crosses the imaginary axis from left to right as $\tau$ increases. Also, the purely imaginary root $i\omega_{2}$ crosses the imaginary axis from left to right if $\sqrt{D}>d$, while $i\omega_{2}$ crosses the imaginary axis from right to left if $\sqrt{D}<d$, where $\omega_{1}$ and $\omega_{2}$ are defined by $(2.1)$ and $(2.2)$.
\end{lem}
\begin{proof}
Taking the derivative of $\lambda$ with respect to $\tau$ on the both sides of (3.3), we obtain 
\begin{equation*}
\frac{d\lambda}{d\tau}-\a e^{-\lambda \tau}\left(\tau \frac{d\lambda}{d\tau}+\lambda \right)=0. 
\end{equation*}
Thus, we have 
\begin{equation*}
\left(1+\tau(\lambda+a+id)\right)\frac{d\lambda}{d\tau}=
-\lambda(\lambda+a+id).
\end{equation*}
Therefore, substituting $\lambda=i\omega$ to the above equation, we have 
\begin{align*}
\left.\frac{d\lambda}{d\tau}\right|_{\lambda=i\omega}&=-\frac{i\omega(i\omega+a+id)}
{1+\tau(i\omega+a+id)}=\frac{\omega(\omega+d)-a \omega i}{(1+a \tau)+i\tau(\omega+d)}=\frac{(\omega(\omega+d)-a \omega i)((1+a \tau)-i\tau(\omega+d))}{(1+a \tau)^2+\tau^2(\omega+d)^2}.
\end{align*}
This yields that 
\begin{equation}
\d\Re\left.\frac{d\lambda}{d\tau}\right|_{\lambda=i\omega}
=\frac{\omega(\omega+d)(1+a \tau)-a \omega \tau(\omega+d)}{{(1+a \tau)^2+\tau^2(\omega+d)^2}}
=\frac{\omega(\omega+d)}{{(1+a \tau)^2+\tau^2(\omega+d)^2}}. 
\end{equation}
If we take $\omega=\omega_{1}$ in (3.5), since $\omega_{1}(\omega_{1}+d)=(d+\sqrt{D})\sqrt{D}>0$, we have $\Re (d\lambda/d\tau)|_{\lambda=i\omega_{1}}>0$. Therefore, the root $i\omega_{1}$ crosses the imaginary axis from left to right as $\tau$ increases. On the other hand, if we take $\omega=\omega_{2}$, we have $\omega_{2}(\omega_{2}+d)=(-d+\sqrt{D})\sqrt{D}$. Therefore, it follows that $\Re (d\lambda/d\tau)|_{\lambda=i\omega_{2}}>0$ holds if $\sqrt{D}>d$, while $\Re (d\lambda/d\tau)|_{\lambda=i\omega_{2}}<0$ holds if $\sqrt{D}<d$. Thus, $i\omega_{2}$ crosses the imaginary axis from left to right if $\sqrt{D}>d$, while $i\omega_{2}$ crosses the imaginary axis from right to left if $\sqrt{D}<d$. This completes the proof. 
\end{proof}
\begin{rem}
{\rm
In the case of $\a^{2}=a^{2}$, since $\omega_{1}=\omega_{2}=-d$, we have $\Re (d\lambda/d\tau)|_{\lambda=i\omega_{j}}=0$ for $j=1,2$ from (3.5). Hence, we cannot analyze the behavior of the roots $i\omega_{j}$ by using (3.5). 
}
\end{rem}

Now, we are ready to complete the proof of Theorem 2.1. 
\begin{prop}
Let $\a \neq0$, $bc<0$ and $\a^{2}>a^{2}$. Then all the roots of $(3.1)$ have negative real parts if and only if the condition $\rm{(ii-1)}$ holds: 
\end{prop}
\begin{proof}
For all $\tau \ge0$, we define $\nu (\tau)$ as the number of the roots of (3.3) including multiplicity whose real parts are positive.

\underline{Case $\rm{(I)}$: $a+\a>0$.} In this case, we note that equation (3.3) has only root $\lambda=-\a-a-id$ for $\tau=0$, and it means all the root lie in the left half of the complex plane for $\tau=0$. Also, from $\a^{2}>a^{2}$ and $a+\a>0$, we have $\a>0$.\\[.5em]
\underline{Subcase $\rm{(I-1)}$: $\sqrt{D}>d$.} From Lemma 3.6, (3.3) has roots $\omega=\omega_{1}<0$ with $\tau=r_{1. n}$ or $\omega=\omega_{2}>0$ with $\tau=r_{2. n}$ for $n=0, 1, 2, \cdots$. Let $\lambda_{1.n}(\tau)$ and $\lambda_{2. n}(\tau)$ are the branches of the roots of (3.3) such that $\lambda_{1.n}(r_{1. n})=i\omega_{1}$ and $\lambda_{2. n}(r_{2. n})=i\omega_{2}$, respectively. By virtue of Corollary 3.4 and Lemma 3.7, $\lambda_{1. n}(\tau)$ moves in the right half of the complex plane and cannot move in the left half of the complex plane crossing on the imaginary axis as $\tau$ increases from $r_{1. n}$. Similarly, $\lambda_{2. n}(\tau)$ moves in the right half-plane and cannot move in the left half-plane crossing on the imaginary axis as $\tau$ increases from $r_{2. n}$. Therefore, since $r_{1. 0}<r_{2. 0}$, it follows that 
\begin{equation}
\nu(\tau)=0 \ \ \text{if} \ \ \tau \in (0,r_{1.0})\ ,\ \nu(\tau)\ge1 \ \ \text{if} \ \ \tau \in (r_{1.0},\infty).
\end{equation}
\underline{Subcase $\rm{(I-2)}$: $\sqrt{D}=d$.} From Lemma 3.6, (3.3) has roots $\omega=\omega_{1}<0$ with $\tau=r_{1. n}$ for $n=0, 1, 2, \cdots$. Let $\lambda_{1.n}(\tau)$ is the branches of the roots of (3.3) such that $\lambda_{1. n}(r_{1.n})=i\omega_{1}$. From Corollary 3.4 and Lemma 3.7, $\lambda_{1.n}(\tau)$ moves in the right half-plane and cannot move in the left half-plane crossing on the imaginary axis as $\tau$ increases from $r_{1. n}$. Therefore, we obtain (3.6). \\[.5em]
\underline{Subcase $\rm{(I-3)}$: $\sqrt{D}<d$.} From Lemma 3.6, (3.3) has roots $\omega=\omega_{1}<0$ with $\tau=r_{1. n}$ or $\omega=\omega_{2}<0$ with $\tau=r_{2. n}$ for $n=0, 1, 2, \cdots$. Let $\lambda_{1.n}(\tau)$ and $\lambda_{2. n}(\tau)$ are the branches of the roots of (3.3) such that $\lambda_{1.n}(r_{1. n})=i\omega_{1}$ and $\lambda_{2. n}(r_{2. n})=i\omega_{2}$, respectively. Corollary 3.4 and Lemma 3.7 asserts that $\lambda_{1. n}(\tau)$ moves in the right half-plane as $\tau$ increases from $r_{1. n}$. On the other hand, $\lambda_{2. n}(\tau)$ moves in the left half-plane as $\tau$ increases from $r_{2. n}$. 
Thus, we need to examine the order relation between $r_{1. n}$ and $r_{2. n}$. First, we note that 
\begin{equation}
r_{2.n}-r_{1.n}=\frac{2}{\omega_1 \omega_2} \left(2n\sqrt{D}\pi+\pi d+\sqrt{D}(\pi-\coss E)\right)>0, \ \ n=0,1,2, \cdots, 
\end{equation}
\begin{equation}
r_{1.n}-r_{2.n-1}=\frac{2}{\omega_1 \omega_2} \left(d\ \coss E-2n\sqrt{D}\pi\right), \ \ r_{1.n+1}-r_{1.n}=\frac{2\pi}{|\omega_1|}<\frac{2\pi}{|\omega_2|}=r_{2.n+1}-r_{2.n}. 
\end{equation}
Therefore, from (3.8), there exists a positive integer $m_{1}(n)$ such that 
\begin{equation}
r_{2.n}<r_{1.m_{1}(n)}<r_{2.n+1}.
\end{equation}
\noindent
\underline{Subcase $\rm{(I-3-a)}$: $d\coss E<2\sqrt{D}\pi$, i.e. $k=0$} (recall $\d k=\left\lfloor\frac{d\ \coss E}{2\sqrt{D}\pi}\right\rfloor$).  \\[.1em]
In this case, we have $r_{1.n}-r_{2.n-1}<0$ for $n=1,2,\cdots$ from (3.8). Thus, together with (3.9), we get 
\begin{equation*}
0<r_{1.0}<r_{1.1}<r_{2.0} <r_{1.2}<\cdots. 
\end{equation*}
Therefore, from Corollary 3.4 and Lemma 3.7, we obtain (3.6). \\[.5em]
\underline{Subcase $\rm{(I-3-b)}$: $d\coss E\ge2\sqrt{D}\pi$, i.e. $k\ge1$.} \\[.3em]
In this case, since 
\begin{equation*}
r_{1.n}-r_{2.n-1}>0, \ \ n=1,2,\cdots , k \ \ \text{and} \ \ r_{1.n}-r_{2.n-1}<0, \ \ n=k+1,k+2,\cdots 
\end{equation*}
from (3.9), it follows from (3.7) that 
\begin{equation}
0<r_{1.0}<r_{2.0}<\cdots<r_{2.k-1}<r_{1.k}<r_{1.k+1}<r_{2.k}<\cdots.
\end{equation}
Therefore, from Corollary 3.4, Lemma 3.7, (3.9) and (3.10), we finally obtain 
\begin{equation*}
\begin{cases}
\nu(\tau)=0\ \ \text{if} \ \ \tau \in (0,r_{1.0}) \cup (r_{2.0},r_{1.1}) \cup \cdots \cup (r_{2.k-1},r_{1.k}), \\
\nu(\tau)=1\ \ \text{if} \ \ \tau \in (r_{1.0},r_{2.0})\cup (r_{1.1},r_{2.1})\cup \cdots \cup(r_{1.k-1},r_{2.k-1}),\\
\nu(\tau)\ge 1\ \ \text{if} \ \ \tau \in (r_{1.k},\infty).
\end{cases}
\end{equation*}

\underline{Case $\rm{(I\hspace{-.1em}I)}$: $a+\a<0$.} In this case, we note that equation (3.3) has only root $\lambda=-\a-a-id$ for $\tau=0$, and it means all the root lie in the right half of the complex plane for $\tau=0$. Also, from $\a^{2}>a^{2}$ and $a+\a<0$, we have $\a<0$.\\[.5em]
\underline{Subcase $\rm{(I\hspace{-.1em}I-1)}$: $\sqrt{D}\ge d$.} By the same argument given in Subcase $\rm{(I-1)}$ and Subcase $\rm{(I-2)}$, the characteristic roots of (3.3) cannot move in the left half-plane crossing the imaginary axis as $\tau$ increases. Therefore, we have $\nu(\tau) \ge1$ for $\tau>0$. \\[.5em]
\underline{Subcase $\rm{(I\hspace{-.1em}I-2)}$: $\sqrt{D}<d$.} Similarly to Subcase $\rm{(I-3)}$, we have the following results: 
\begin{equation*}
\begin{cases}
\nu(\tau)=0 \ \ \text{if} \ \ \tau \in (r_{2.0},r_{1.0}) \cup (r_{2.1},r_{1.1}) \cup \cdots \cup (r_{2.l},r_{1.l}), \\
\nu(\tau)=1\ \ \text{if} \ \ \tau \in (0,r_{2.0})\cup(r_{1.0},r_{2.1})\cup  \cdots \cup(r_{1.l-1},r_{2.l}),\\
\nu(\tau)\geq 1\ \ \text{if} \ \ \tau \in (r_{1.l},\infty).
\end{cases}
\end{equation*}
 
In conclusion, if $\a \neq0$, $bc<0$ and $\a^{2}>a^{2}$, all the roots of $(3.3)$ have negative real parts if and only if the condition $\rm{(ii-1)}$ holds. This completes the proof.   
\end{proof}

Finally, we can prove Theorem 2.1. 
\begin{proof}[\rm{\bf{End of the Proof of Theorem 2.1}}]
Combining Proposition 3.2, Proposition 3.5 and Proposition 3.9, we immediately have Theorem 2.1. 
\end{proof}

\vskip10pt
\par\noindent
\textbf{Conclusion} 

We have established the necessary and sufficient conditions for the asymptotic stability for the zero solution of (1.1) including the stability switching. We believe that our result can be applicable for the nonlinear problems in the several fields like population models, neural network models, and also traffic flow, to investigate some dynamical behavior.

\vskip10pt
\par\noindent
\textbf{Acknowledgments} 

The authors would like to express their sincere gratitude to Professor Hideo Kubo for his feedback and valuable advices. The authors also would like to thank Professor Yoshihiro Ueda for his useful suggestions and comments.

This study is partially supported by MEXT through Program for Leading Graduate Schools (Hokkaido University ``Ambitious Leader's Program"). 


\vskip10pt
\par\noindent
\begin{flushleft}Wataru Saito\\
Department of Mathematics, Hokkaido University\\
Sapporo 060-0810, Japan\\
E-mail: s173012math@frontier.hokudai.ac.jp
\vskip10pt
Ikki Fukuda\\
Department of Mathematics, Hokkaido University\\
Sapporo 060-0810, Japan\\
E-mail: i.fukuda@math.sci.hokudai.ac.jp
\end{flushleft}

\end{document}